\documentclass[12pt]{amsart}
\usepackage[normalem]{ulem}

\usepackage{verbatim}
\usepackage{amssymb}
\usepackage{upref}
\usepackage[all]{xy}
\usepackage{color}
\usepackage{url}

\allowdisplaybreaks

\newtheorem{thm}{Theorem}[section]
\newtheorem*{thm*}{Theorem}
\newtheorem{lem}[thm]{Lemma}
\newtheorem{cor}[thm]{Corollary}
\newtheorem{prop}[thm]{Proposition}

\theoremstyle{definition}
\newtheorem{defn}[thm]{Definition}

\newtheorem*{notn*}{Notation}

\newtheorem*{hyp*}{Hypothesis}

\newtheorem{rem}[thm]{Remark}
\newtheorem*{rem*}{Remark}
\newtheorem{rems}[thm]{Remarks}

\numberwithin{equation}{section}

\newcommand{\secref}[1]{Section~\textup{\ref{#1}}}

\newcommand{\thmref}[1]{Theorem~\textup{\ref{#1}}}
\newcommand{\corref}[1]{Corollary~\textup{\ref{#1}}}
\newcommand{\lemref}[1]{Lemma~\textup{\ref{#1}}}
\newcommand{\propref}[1]{Proposition~\textup{\ref{#1}}}

\newcommand{\remref}[1]{Remark~\textup{\ref{#1}}}

\newcommand{\midtext}[1]{\quad\text{#1}\quad}

\renewcommand{\and}{\midtext{and}}

\newcommand{\Z}{\mathbb Z}

\newcommand{\T}{\mathbb T}

\newcommand{\KK}{\mathcal K}

\newcommand{\LL}{\mathcal L}

\newcommand{\OO}{\mathcal O}

\renewcommand{\epsilon}{\varepsilon}

\DeclareMathOperator{\ad}{Ad}

\DeclareMathOperator*{\spn}{span}
\DeclareMathOperator*{\clspn}{\overline{\spn}}

\newcommand{\id}{\text{\textup{id}}}

\newcommand{\<}{\langle}
\renewcommand{\>}{\rangle}

\newcommand{\inv}{^{-1}}

\newcommand{\iso}{\overset{\cong}{\longrightarrow}}
\renewcommand{\bar}{\overline}
\newcommand{\what}{\widehat}

\newcommand{\rt}{\textup{rt}}

\newcommand{\cs}%
{\ensuremath{\mathbf{C^*}}}
\newcommand{\csnd}%
{\ensuremath{\cs\!\!_\mathbf{nd}}}
\newcommand{\coact}%
{\ensuremath{\mathbf{C^*coact}}}
\newcommand{\coactnd}%
{\ensuremath{\coact_\mathbf{nd}}}
\newcommand{\coactn}%
{\ensuremath{\coact^\mathbf{n}}}
\newcommand{\coactnnd}%
{\ensuremath{\coactn_\mathbf{nd}}}
\newcommand{\coactm}%
{\ensuremath{\coact^\mathbf{m}}}
\newcommand{\coactmnd}%
{\ensuremath{\coactm_\mathbf{nd}}}

\newcommand{\cpct}[1]{#1^{(1)}}
\newcommand{\ideal}[2]{M(#1;#2)}
\newcommand{\cg}{C^*(G)}

\begin{document}
\title{Coactions on Cuntz-Pimsner algebras}
\author[Kaliszewski, Quigg, and Robertson]
{S.~Kaliszewski, John Quigg, and David Robertson}
\address [S.~Kaliszewski] {School of Mathematical and Statistical Sciences, Arizona State University, Tempe, Arizona 85287} \email{kaliszewski@asu.edu}\
\address[John Quigg]{School of Mathematical and Statistical Sciences, Arizona State University, Tempe, Arizona 85287} \email{quigg@asu.edu}
\address[David Robertson]{School of Mathematics and Applied Statistics, University of Wollongong, NSW 2522, AUSTRALIA} \email{droberts@uow.edu.au}

\date{\today}

\begin{abstract}
We investigate how a correspondence coaction gives rise to a coaction on the associated Cuntz-Pimsner algebra.
We apply this to recover a recent result of Hao and Ng concerning Cuntz-Pimsner algebras of crossed products of correspondences by actions of amenable groups.
\end{abstract}


\maketitle

\section{Introduction}\label{intro}

The Cuntz-Pimsner algebra $\OO_X$ associated to a $C^*$-correspondence $X$ is a $C^*$-algebra whose representations encode the Cuntz-Pimsner covariant representations of $X$.
These were introduced by Pimsner in \cite{Pi}, and generalize both crossed products by $\Z$ and graph algebras when the underlying graph has no sources. Further work by Katsura in \cite{KatsuraCorrespondence} has expanded the class of Cuntz-Pimsner algebras to include graph algebras of arbitrary graphs, crossed products by partial automorphisms and topological graph algebras.

As in the cases of the above mentioned $C^*$-algebras, it is fruitful to investigate how $C^*$-constructions involving $\OO_X$ can be studied in terms of corresponding constructions involving $X$.
For example, it has been understood for some time how actions of groups on $\OO_X$ can be studied in terms of actions on $X$, see \cite{HN} for example.
In this paper we show how coactions of a locally compact group $G$ on $\OO_X$ can be studied in terms of suitable coactions of $G$ on $X$.

In order to say what ``suitable'' should mean, we appeal to \cite{KQRCorrespondenceFunctor}, where we showed that the passage from $X$ to $\OO_X$ is functorial for certain categories.
Specifically, the target category is $C^*$-algebras and nondegenerate homomorphisms into multiplier algebras, and the domain category is correspondence and \emph{Cuntz-Pimsner covariant homomorphisms} (defined in \cite{KQRCorrespondenceFunctor}).
To see how this should be applied, note that a coaction of $G$ on $\OO_X$ is a nondegenerate homomorphism $\zeta:\OO_X\to M(\OO_X\otimes C^*(G))$ satisfying appropriate conditions, and similarly a coaction of $G$ on $X$ (as defined in \cite{enchilada}) is a homomorphism $\sigma:X\to M(X\otimes C^*(G))$.
In order to apply the techniques from \cite{KQRCorrespondenceFunctor}, we want $\zeta$ to be determined by $\sigma$.
If we knew that $\OO_X\otimes C^*(G)$ were equal to $\OO_{X\otimes C^*(G)}$, the Cuntz-Pimsner algebra of the external-tensor-product correspondence, then the main result of \cite{KQRCorrespondenceFunctor} would tell us that 
we should require the correspondence homomorphism $\sigma$ to be \emph{Cuntz-Pimsner covariant} in the sense defined there.
As it happens, 
due to the nonexactness of minimal $C^*$-tensor products,
we need a slightly stronger version of Cuntz-Pimsner covariance, specifically suited for correspondence coactions.
We work this out in an abstract setting toward the end of \secref{prelim},
then we use this to prove out main result concerning coactions on Cuntz-Pimsner algebras
at the start of \secref{coactions}, after which we go on to develop a few tools dealing with inner coactions on correspondences. 

In \secref{crossed products} we show how to recognize covariant representations of the coaction $\zeta$ on $\OO_X$ using the coaction $\sigma$ on $X$.
In \thmref{crossed} we show that under a mild technical condition the crossed product $\OO_X\rtimes_\zeta G$ is isomorphic to the Cuntz-Pimsner algebra $\OO_{X\rtimes_\sigma G}$ of the crossed-product correspondence. We 
list in \lemref{implies CP} a couple of situations in which the technical condition is guaranteed to hold.
We also show that, as in the $C^*$-case, the crossed product of $X$ by an inner coaction is isomorphic to the tensor product $X\otimes C_0(G)$,
and that if $G$ is amenable and acts on $X$ then the dual coaction on the crossed product $X\rtimes G$ 
satisfies our stronger version of Cuntz-Pimsner covariance. For all we know the amenability hypothesis in the latter result is unnecessary, but anyway we will apply this 
in \secref{apps} to recover a recent result of Hao and Ng \cite{HN}; they show that if $G$ acts on $X$ then $\OO_X\rtimes G\cong \OO_{X\rtimes G}$, and we give a substantially different proof using the techniques of the present paper.

\section{Preliminaries}\label{prelim}

We are mainly interested in correspondences over a single coefficient $C^*$-algebra, but occasionally we will find it convenient to allow the left and right coefficient $C^*$-algebras to be different.
We denote an $A-B$ correspondence $X$ by $(A,X,B)$ and write $\phi_A: A \to \LL(X)$ for the left action of $A$ on $X$.
If $A=B$ we denote the $A$-correspondence $X$ by $(X,A)$.
All correspondences will be assumed \emph{nondegenerate} in the sense that $A\cdot X=X$.\footnote{Warning: in \cite{enchilada} the definition of ``right-Hilbert bimodule'' includes the nondegeneracy hypothesis.}
We record here the notation and results that we will need.

The \emph{multiplier correspondence} of a correspondence $(A,X,B)$ is
$M(X):=\LL_B(B,X)$,
which is an $M(A)-M(B)$ correspondence in a natural way.
If $(A,X,B)$ and $(C,Y,D)$ are correspondences, a \emph{correspondence homomorphism}
$(\pi,\psi,\rho):(A,X,B)\to (M(C),M(Y),M(D))$
comprises homomorphisms $\pi:A\to M(C)$ and $\rho:B\to M(D)$ and a linear map $\psi:X\to M(Y)$
preserving the correspondence operations.
The homomorphism $(\pi,\psi,\rho)$ is \emph{nondegenerate} if $\clspn\{\psi(X)\cdot D\}=Y$ and both $\pi$ and $\rho$ are nondegenerate,
and then there is a unique strictly continuous extension
$(\bar\pi,\bar\psi,\bar\rho):(M(A),M(X),M(B))\to (M(C),M(Y),M(D))$,
and also a unique nondegenerate homomorphism
$\psi^{(1)}:\KK(X)\to \LL(Y)$
such that
$(\psi^{(1)},\psi,\rho):(\KK(X),X,B)\to (\LL(Y),M(Y),M(D))$
is a nondegenerate correspondence homomorphism.
The diagram
\begin{equation}\label{nondegenerate commute}
\xymatrix{
A \ar[r]^-\pi \ar[d]_{\varphi_A}
&M(B) \ar[d]^{\bar{\varphi_B}}
\\
\LL(X) \ar[r]_-{\bar{\psi^{(1)}}}
&\LL(Y)
}
\end{equation}
commutes, and $\psi^{(1)}$ is determined by
$\psi^{(1)}(\theta_{\xi,\eta})=\psi(\xi)\psi(\eta)^*$.

If $A=B$, $C=D$, and $\pi=\rho$, we write
$(\psi,\pi):(X,A)\to (M(Y),M(C))$.

We refer to \cite[Section~2]{KQRCorrespondenceFunctor} for an exposition of the properties of  the ``relative multipliers'' from \cite[Appendix~A]{dkq}.
Very briefly, 
if $(X,A)$ is a nondegenerate correspondence and $\kappa:C\to M(A)$ is a nondegenerate homomorphism, the
\emph{$C$-multipliers} of $X$ are
\[
M_C(X):=\{m\in M(X):\kappa(C)\cdot m\cup m\cdot \kappa(C)\subset X\}.
\]
The main purpose of relative multipliers is the following extension theorem \cite[Proposition~A.11]{dkq}:
let $X$ and $Y$ be nondegenerate correspondences over $A$ and $B$, respectively,
let $\kappa:C\to M(A)$ and
$\sigma:D\to M(B)$
be nondegenerate homomorphisms.
If there is a nondegenerate homomorphism $\lambda:C\to M(\sigma(D))$
such that
\[
\pi(\kappa(c)a)=\lambda(c)\pi(a)\midtext{for}c\in C,a\in A,
\]
then for any correspondence homomorphism $(\psi,\pi):(X,A)\to (M_D(Y),M_D(B))$ there is a unique $C$-strict to $D$-strictly continuous correspondence homomorphism $(\bar\psi,\bar\pi)$ making the diagram
\[
\xymatrix{
(X,A) \ar[r]^-{(\psi,\pi)} \ar@{^(->}[d]
&(M_D(Y),M_D(B))
\\
(M_C(X),M_C(A)) \ar@{-->}[ur]_{(\bar\psi,\bar\pi)}^{!}
}
\]
commute.

We will also need to use the method of \cite{KQRCorrespondenceFunctor} to construct homomorphisms of Cuntz-Pimsner algebras from correspondence homomorphisms:
a homomorphism $(\psi,\pi):(X,A)\to (M(Y),M(B))$ is \emph{Cuntz-Pimsner covariant} if
\begin{enumerate}
\item $\psi(X)\subset M_B(Y)$,

\item $\pi:A\to M(B)$ is nondegenerate,

\item $\pi(J_X)\subset \ideal{B}{J_Y}$, and

\item the diagram
\begin{equation}\label{CP diagram}
\xymatrix{
J_X \ar[r]^-{\pi|} \ar[d]_{\varphi_A|}
&\ideal{B}{J_Y} \ar[d]^{\bar{\varphi_B}\bigm|}
\\
\KK(X) \ar[r]_-{\cpct\psi}
&M_B(\KK(Y))
}
\end{equation}
commutes,
\end{enumerate}
where, for an ideal $I$ of a $C^*$-algebra $A$, we follow \cite{BaajSkandalis} by defining
\[
\ideal{A}{I}=\{m\in M(A):mA\cup Am\subset I\}.
\]
By \cite[Corollary~3.6]{KQRCorrespondenceFunctor}, when $(\psi,\pi)$ is Cuntz-Pimsner covariant there is a unique homomorphism $\OO_{\psi,\pi}$ making the diagram
\[
\xymatrix{
X \ar[r]^-{(\psi,\pi)} \ar[d]_{k_X}
&M_B(Y) \ar[d]^{\bar{k_Y}}
\\
\OO_X \ar[r]_-{\OO_{\psi,\pi}}
&M_B(\OO_Y)
}
\]
commute.

If $G$ is a locally compact group and
$(X,A)$ is a correspondence we will write
\begin{align*}
M_{\cg}(A\otimes \cg)&=M_{1\otimes \cg}(A\otimes \cg)\\
M_{\cg}(X\otimes \cg)&=M_{1\otimes \cg}(X\otimes \cg).
\end{align*}

Recall that a \emph{coaction} of $G$ on a $C^*$-algebra $A$ is a 
nondegenerate 
injective homomorphism $\delta:A\to M(A\otimes \cg)$ satisfying the \emph{coaction identity} given by the commutative diagram
\begin{equation}\label{coaction diagram}
\xymatrix@C+30pt{
A \ar[r]^-\delta \ar[d]_\delta
&M(A\otimes \cg) \ar[d]^{\bar{\delta\otimes\id}}
\\
M(A\otimes \cg) \ar[r]_-{\bar{\id\otimes\delta_G}}
&M(A\otimes \cg\otimes \cg),
}
\end{equation}
and satisfying the \emph{coaction-nondegeneracy} condition
\[
\clspn\{\delta(A)(1\otimes \cg)\}=A\otimes \cg.
\]

\begin{rems}
(1)
Note that, as has become customary in recent years, we have built coaction-nondegeneracy into the definition of coaction, and of course it follows that $\delta(A)\subset M_{\cg}(A\otimes \cg)$.

(2)
The coaction identity requires $\delta$ to be nondegenerate as a homomorphism, so that it extends uniquely to multipliers.\footnote{However, if we know that $\delta(A)\subset M_{\cg}(A\otimes \cg)$, then, even without knowing $\delta$ is nondegenerate, the coaction identity makes sense when the upper right and lower left corners of the commutative diagram \eqref{coaction diagram} are replaced by $M_{\cg}(A\otimes \cg)$.}

(3)
Coaction-nondegeneracy implies nondegeneracy as a homomorphism.
However, an under-appreciated result of Katayama \cite[Lemma~4]{kat},
implies that, assuming we know $\delta$ satisfies all the other coaction axioms except for coaction-nondegeneracy, the closed span of the products $\delta(A)(1\otimes \cg)$ is actually a $C^*$-subalgebra of $A\otimes \cg$,
and hence to show coaction-nondegeneracy it suffices to verify 
the seemingly weaker condition
\begin{equation}\label{katayama}
\text{$\delta(A)(1\otimes \cg)$ generates 
$A\otimes \cg$ as a $C^*$-algebra.}
\end{equation}
\end{rems}

A nondegenerate homomorphism
$\mu:C_0(G)\to M(A)$ implements an \emph{inner coaction} $\delta^\mu$ on $A$ via
\[
\delta^\mu(a)=\ad\bar{\mu\otimes\id}(w_G)(a\otimes 1),
\]
where
\[
w_G\in M(C_0(G)\otimes \cg)=C_b(G,M^\beta(C^*(G)))
\]
is the function given by the canonical embedding of $G$ into the unitary group of $M(C^*(G))$.
The \emph{trivial coaction} $\delta^1=\id_A\otimes 1$ on $A$ is implemented by the homomorphism
\[
f\mapsto f(e)1_{M(A)}\midtext{for}f\in C_0(G).
\]

A coaction $(A,\delta)$ makes $A$ into a Banach module over the Fourier-Stieltjes algebra $B(G)=C^*(G)^*$ via
\[
f\cdot a=S_f\circ\delta(a)\midtext{for}f\in B(G),a\in A,
\]
where $S_f:A\otimes C^*(G)\to A$ is the slice map, which we sometimes alternatively denote by $\id\otimes f$.
Frequently we restrict 
the module action to the Fourier algebra $A(G)$, which is dense in $C_0(G)$.

The \emph{Kronecker product} (see, e.g., \cite[Th\'eor\`eme~1.5]{DeCanniereEnockSchwartz}, \cite[p. 118]{KirchbergHopf}, \cite[Definition~A.2]{NakagamiTakesaki}, \cite[Definition~6.6]{QuiggDualityTwisted}) of two nondegenerate homomorphisms $\mu$ and $\nu$ of $C_0(G)$ in $M(A)$ and $M(B)$, respectively, is defined by
\[
\mu\times\nu:=\bar{\mu\otimes\nu}\circ\alpha,
\]
where $\alpha:C_0(G)\to C_b(G\times G)=M(C_0(G)\otimes C_0(G))$ is given by
\[
\alpha(f)(s,t)=f(st).
\]
Letting
\[
u=\bar{\mu\otimes\id}(w_G)\midtext{and}v=\bar{\nu\otimes\id}(w_G),
\]
we have
\[
\bar{(\mu\times\nu)\otimes\id}(w_G)=u_{13}v_{23}.
\]

A \emph{covariant homomorphism} of a coaction $(A,\delta)$
is a pair
$(\pi,\mu):(A,C_0(G))\to M(B)$ comprising nondegenerate homomorphisms $\pi:A\to M(B)$ and $\mu:C_0(G)\to M(B)$ such that
\[
\bar{\pi\otimes\id}\circ\delta(a)=\ad\bar{\mu\otimes\id}(w_G)(\pi(a)\otimes 1).
\]
A \emph{crossed product} of $(A,\delta)$ is a triple $(A\rtimes_\delta G,j_A,j_G)$ consisting of a covariant homomorphism $(j_A,j_G):(A,C_0(G))\to M(A\rtimes_\delta G)$ that is \emph{universal} in the sense that for every covariant homomorphism $(\pi,\mu):(A,C_0(G))\to M(B)$ there is a unique nondegenerate homomorphism $\pi\times\mu:A\rtimes_\delta G\to M(B)$ making the diagram
\[
\xymatrix{
A \ar[r]^-{j_A} \ar[dr]_\pi
&M(A\rtimes_\delta G) \ar@{-->}[d]^{\pi\times\mu}_{!}
&C_0(G) \ar[l]_-{j_G} \ar[dl]^\mu
\\
&M(B)
}
\]
commute.
It follows that $A\rtimes_\delta G=\clspn\{j(A)j_G(C_0(G))$.
The crossed product is unique up to isomorphism,
and one construction is given by the \emph{regular representation}
\[
\bigl((\id\otimes\lambda)\circ\delta,1\otimes M\bigr):
(A,C_0(G))\to M(A\otimes \KK(L^2(G))),
\]
where $\lambda$ is the left regular representation of $G$ and
$M:C_0(G)\to B(L^2(G))$ is the multiplication representation.

For correspondence coactions, we follow
\cite{enchilada}, 
but again build in coaction-nondegeneracy:

\begin{defn}
A \emph{coaction} of $G$ on a correspondence $(A,X,B)$ is a 
nondegenerate 
correspondence homomorphism
\[
(\delta,\sigma,\epsilon):(A,X,B)\to \bigl(M(A\otimes \cg),M(X\otimes \cg),M(B\otimes \cg)\bigr)
\]
such that:
\begin{enumerate}
\item $\delta$ and $\epsilon$ are coactions on $A$ and $B$, respectively;

\item $\sigma$ satisfies the coaction identity given by the commutative diagram
\[
\xymatrix@C+30pt{
X \ar[r]^-\sigma \ar[d]_\sigma
&M(X\otimes \cg) \ar[d]^{\bar{\sigma\otimes\id}}
\\
M(X\otimes \cg) \ar[r]_-{\bar{\id\otimes\delta_G}}
&M(X\otimes \cg\otimes \cg);
}
\]

\item $\sigma$ satisfies the coaction-nondegeneracy condition
\[
\clspn\{(1\otimes \cg)\cdot \sigma(X)\}=\clspn\{\sigma(X)\cdot (1\otimes \cg)\}=X\otimes \cg.
\]
\end{enumerate}
We also say that $\sigma$ is \emph{$\delta-\epsilon$ compatible}.
\end{defn}

\begin{rems}
(1)
Remarks similar to those following the definition of $C^*$-coaction apply to correspondence coactions. For example,
coaction-nondegeneracy implies that $\sigma(X)\subset M_{\cg}(X\otimes \cg)$ and $\sigma$ is nondegenerate as a correspondence homomorphism.
In fact, it implies a stronger form of nondegeneracy, namely that, in addition to $\clspn\{\sigma(X)\cdot (X\otimes \cg)\}=X\otimes \cg$, we also have the symmetric property on the other side:
\[
\clspn\{(X\otimes \cg)\cdot \sigma(X)\}=X\otimes \cg.
\]

(2)
On the other hand, nondegeneracy of $\sigma$ as a correspondence homomorphism implies one half of the coaction-nondegeneracy, namely $\clspn\{\sigma(X)\cdot (1\otimes \cg)\}=X\otimes \cg$, by coaction-nondegeneracy of $\epsilon$.

(3)
$\sigma$ will be isometric since $\epsilon$ is injective.
\end{rems}

Frequently we will have $A=B$ and $\delta=\epsilon$, in which case we say that $(\sigma,\delta)$ is a coaction on $(X,A)$;
of course the case $X=A=B$ and $\sigma=\delta=\epsilon$ reduces to a $C^*$-coaction.
Being particularly nice correspondence homomorphisms, coactions on $C^*$-correspondences are easily shown to be Cuntz-Pimsner covariant:

\begin{lem}\label{coaction CP}
A coaction $(\sigma,\delta)$ of $G$ on a correspondence $(X,A)$ is Cuntz-Pimsner covariant as a correspondence homomorphism if and only if 
\[
\delta(J_X)\subset M\bigl(A\otimes \cg;J_{X\otimes \cg}\bigr).
\]
\end{lem}

\begin{proof}
By definition of correspondence coaction, the correspondence homomorphism $(\sigma,\delta):(X,A)\to (M(X\otimes \cg),M(A\otimes \cg))$ is nondegenerate,
and
the inclusion $\sigma(X)\subset M_{\cg}(X\otimes \cg)$ trivially implies that $\sigma(X)\subset M_{A\otimes \cg}(X\otimes \cg)$.
Combining with \cite[Lemma 3.2]{KQRCorrespondenceFunctor} gives the result.
\end{proof}

However, as consequence of the nonexactness of minimal $C^*$-tensor products, we will need a variation on \lemref{coaction CP}, and we state it in abstract form, not involving coactions:

\begin{lem}\label{tensor}
Let
$(X,A)$ be a correspondence,
let $C$ be a $C^*$-algebra,
and
let $(\psi,\pi):(X,A)\to (M_C(X\otimes C),M_C(A\otimes C))$ be a nondegenerate correspondence homomorphism.
If
\[
\pi(J_X)\subset M(A\otimes C;J_X\otimes C),
\]
then the composition
\[
\Bigl(\bar{k_X\otimes\id}\circ\psi,\bar{k_A\otimes\id}\circ\pi\Bigr):
(X,A)\to M(\OO_X\otimes C)
\]
is Cuntz-Pimsner covariant.
\end{lem}

\begin{proof}
By checking on elementary tensors one verifies that, on the ideal $J_X\otimes C$ of $A\otimes C$, we have
\begin{align*}
(k_X\otimes\id)^{(1)}\circ \varphi_{A\otimes C}
&=(k_X^{(1)}\otimes\id)\circ (\varphi_A\otimes\id)
\\&=k_X^{(1)}\circ \varphi_A\otimes \id
\\&=k_A\otimes\id,
\end{align*}
and so, by strict continuity, on $M(A\otimes C;J_X\otimes C)$ we have
\[
\bar{(k_X\otimes\id)^{(1)}}\circ \bar{\varphi_{A\otimes C}}=\bar{k_A\otimes\id}.
\]
Thus, on $J_X$ we have
\begin{align*}
\Bigl(\bar{k_X\otimes\id}\circ\psi\Bigr)^{(1)}\circ\varphi_A
&=\bar{(k_X\otimes\id)^{(1)}}\circ\psi^{(1)}\circ\varphi_A
\\&=\bar{(k_X\otimes\id)^{(1)}}\circ\bar{\varphi_{A\otimes C}}\circ\pi,
\\\intertext{by \cite[Lemma~3.3]{KQRCorrespondenceFunctor}, since $(\psi,\pi)$ is nondegenerate,}
\\&=\bar{k_A\otimes\id}\circ\pi,
\end{align*}
which is Cuntz-Pimsner covariance.
\end{proof}

Here is the connection between Lemmas~\ref{coaction CP} and \ref{tensor}:

\begin{lem}\label{contained}
Let $(X,A)$ be a correspondence, let $C$ be a $C^*$-algebra, and let $(X\otimes C,A\otimes C)$ be the external-tensor-product correspondence. Then
\[
J_X\otimes C\subset J_{X\otimes C},
\]
with equality if $C$ is exact.
\end{lem}

\begin{proof}
We use the characterization \cite[Paragraph following Definition~2.3]{KatsuraCorrespondence} of 
$J_X$ as the largest ideal of $A$ that $\varphi_A$ maps injectively into $\KK(X)$, and similarly for $J_{X\otimes C}$.
By \cite[Corollary~3.38]{tfb}, for example, we have
\[
\KK(X\otimes C)=\KK(X)\otimes C,
\]
so
\[
\varphi_{A\otimes C}=\varphi_A\otimes \id_C.
\]
Since $\varphi_A$ maps $J_X$ injectively into $\KK(X)$,
$\varphi_A\otimes\id$ maps $J_X\otimes C$ injectively into $\KK(X)\otimes C$.
Therefore $\varphi_{A\otimes C}$ maps $J_X\otimes C$ injectively into $\KK(X\otimes C)$,
so $J_X\otimes C\subset J_{X\otimes C}$.

Now assume that $C$ is exact,
and let $x\in J_{X\otimes C}$.
Since $C$ is exact, it has the slice map property, so to show that $x\in J_X\otimes C$ it suffices to show that $(\id\otimes\omega)(x)\in J_X$ for all $\omega\in C^*$.
To verify the first property of $J_X$, we have
\begin{align*}
\varphi_A\bigl((\id\otimes\omega)(x)\bigr)
&=(\id\otimes\omega)\circ (\varphi_A\otimes\id)(x),
\end{align*}
which is in $\KK(X)$ because
\[
(\varphi_A\otimes\id)(x)
=\varphi_{A\otimes C}(x)
\in \KK(X\otimes C)
=\KK(X)\otimes C.
\]
For the other property of $J_X$, let $a\in \ker\phi_A$.
Factor $\omega=c\cdot \omega'$ with $c\in C$ and $\omega'\in C^*$. 
Then
\begin{align*}
\bigl((\id\otimes\omega)(x)\bigr)a
&=(\id\otimes c\cdot \omega')\bigl(x(a\otimes 1)\bigr)
\\&=(\id\otimes\omega')\bigl(x(a\otimes c)\bigr),
\end{align*}
which is $0$ because
\[
a\otimes c\in \ker\varphi_A\otimes C=\ker\varphi_{A\otimes C}.
\qedhere
\]
\end{proof}

Recall from \cite[Proposition~3.9]{enchilada} that if $(\delta,\sigma,\epsilon)$ is a coaction of $G$ on a correspondence $(A,X,B)$, then the \emph{crossed product} correspondence 
$(A\rtimes_\delta G,X\rtimes_\sigma G,B\rtimes_\epsilon G)$ is defined by
\[
X\rtimes_\sigma G=\clspn\{j_X(X)\cdot j_G^B(C_0(G))\}
\subset M(X\otimes \KK(L^2(G))),
\]
where
\begin{align*}
j_X&=(\id\otimes\lambda)\circ\sigma\\
j_G^B&=1_{M(B)}\otimes M.
\end{align*}
$X\rtimes_\sigma G$ is an $A\rtimes_\delta G-B\rtimes_\epsilon G$ correspondence in a natural way when we use the regular representations
\begin{align*}
(j_A,j^A_G):(A,C_0(G))&\to M(A\otimes \KK(L^2(G)))\\
(j_B,j^B_G):(B,C_0(G))&\to M(B\otimes \KK(L^2(G))).
\end{align*}

\cite[Lemma~3.10]{enchilada} proves that there is a coaction $\mu$ of $G$ on $\KK(X)$ such that
\begin{itemize}
\item $\varphi_A:A\to M(\KK(X))$ is $\delta-\mu$ equivariant;

\item there is an isomorphism $\KK(X\rtimes_\sigma G)\cong \KK(X)\rtimes_\mu G$
that carries $\varphi_{A\rtimes_\delta G}$ to $\varphi_A\rtimes G$.
\end{itemize}
In fact, the an examination of the construction used in \cite{enchilada} reveals that the coaction on $\KK(X)$ is none other than
\[
\sigma^{(1)}:\KK(X)\to M_{C^*(G)}(\KK(X)\otimes C^*(G))
=M_{C^*(G)}\bigl(\KK(X\otimes C^*(G)\bigr),
\]
so that the left-module action of $A\rtimes_\delta G$ on $X\rtimes_\sigma G$ can be regarded as
\[
\varphi_A\rtimes G:A\rtimes_\delta G\to M\bigl(\KK(X)\rtimes_{\sigma^{(1)}} G\bigr).
\]

\begin{rem}\label{j nondegenerate}
Note that
\[
(j_A,j_X,j_B):(A,X,B)\to \bigl(M(A\rtimes_\delta G),M(X\rtimes_\sigma G),M(B\rtimes_\epsilon G)\bigr)
\]
is a correspondence homomorphism. In fact, it is a bit more: since $j_A$ and $j_B$ are nondegenerate by the standard theory of $C^*$-coactions, it follows from \cite[Lemma~3.10]{enchilada} that he correspondence homomorphism $(j_A,j_X,j_B)$ is nondegenerate.
\end{rem}

\begin{lem}\label{j CP}
Let $(\sigma,\delta)$ be a coaction of $G$ on a correspondence $(X,A)$. Then the canonical correspondence homomorphism $(j_X,j_A):(X,A)\to (M(X\rtimes_\sigma G),M(A\rtimes_\delta G))$ is Cuntz-Pimsner covariant if and only if
\[
j_A(J_X)\subset M(A\rtimes_\delta G;J_{X\rtimes_\sigma G}).
\]
\end{lem}

\begin{proof}
By \remref{j nondegenerate} and \cite[Lemma 3.2]{KQRCorrespondenceFunctor}, it suffices to observe that
\[
j_X(X)\subset M_{A\rtimes_\delta G}(X\rtimes_\sigma G).
\qedhere
\]
\end{proof}

Although the following concept does not appear in \cite{enchilada}, we will find it useful:

\begin{defn}
Let $(\delta,\sigma,\epsilon)$ be a coaction of $G$ on a correspondence $(A,X,B)$,
let $(\pi,\psi,\rho):(A,X,B)\to (M(D),M(Y),M(E))$ be a correspondence homomorphism,
and let $\mu:C_0(G)\to M(D)$ and $\nu:C_0(G)\to M(E)$ be homomorphisms.
Then $(\pi,\psi,\rho,\mu,\nu)$ is \emph{covariant for $(\delta,\sigma,\epsilon)$} if
\begin{enumerate}
\item
$(\pi,\mu)$ and $(\rho,\nu)$ are covariant for $(A,\delta)$ and $(B,\epsilon)$, respectively;

\item
for all $\xi\in X$ we have
\[
\bar{\psi\otimes\id}\circ\sigma(\xi)=
\bar{\mu\otimes\id}(w_G)\cdot \bigl(\psi(\xi)\otimes 1\bigr)\cdot \bar{\nu\otimes\id}(w_G)^*.
\]
\end{enumerate}
\end{defn}

\begin{rem}
Note that covariance of $(\pi,\mu)$ and $(\rho,\nu)$ entails that $\pi,\mu,\rho,\nu$ are all nondegenerate.
\end{rem}

If $A=B$, $\delta=\epsilon$, $\pi=\rho$, and $\mu=\nu$, we say $(\psi,\pi,\mu)$ is covariant for $(\sigma,\delta)$.

\section{Coactions on Cuntz-Pimsner algebras}\label{coactions}

\begin{prop}\label{coaction}
Let $(\sigma,\delta)$ be a coaction of $G$ on a correspondence $(X,A)$.
If
\[
\delta(J_X)\subset M(A\otimes C^*(G);J_X\otimes C^*(G)),
\]
then there is a unique coaction $\zeta$ of $G$ on $\OO_X$ 
making the diagram
\[
\xymatrix@C+30pt{
(X,A) \ar[r]^-{(\sigma,\delta)} \ar[d]_{(k_X,k_A)}
&(M_{\cg}(X\otimes \cg),M_{\cg}(A\otimes \cg)) \ar[d]^{(\bar{k_X\otimes\id},\bar{k_A\otimes\id})}
\\
\OO_X \ar@{-->}[r]_-\zeta^-{!}
&M_{\cg}(\OO_X\otimes \cg)
}
\]
commute.
\end{prop}

\begin{proof}
By definition of correspondence coaction, the correspondence homomorphism $(\sigma,\delta)$ is nondegenerate,
and so, by \lemref{tensor}, our hypothesis guarantees that the composition
\[
\Bigl(\bar{k_X\otimes\id}\circ\sigma,\bar{k_A\otimes\id}\circ\delta\Bigr)
\]
is Cuntz-Pimsner covariant.
Thus
there is a unique homomorphism $\zeta$ making the diagram commute, and moreover $\zeta$ is injective because $\delta$ is.

For the coaction identity, we have
\begin{align*}
\bar{\zeta\otimes\id}\circ\zeta\circ k_X
&=\bar{\zeta\otimes\id}\circ\zeta_X
\\&=\bar{\zeta\otimes\id}\circ\bar{k_X\otimes\id}\circ\sigma
\\&=\bar{\zeta\circ k_X\otimes\id}\circ\sigma
\\&=\bar{\bar{k_X\otimes\id}\circ\sigma\otimes\id}\circ\sigma
\\&=\bar{k_X\otimes\id\otimes\id}\circ\bar{\sigma\otimes\id}\circ\sigma
\\&=\bar{k_X\otimes\id\otimes\id}\circ\bar{\id\otimes\zeta_G}\circ\sigma
\\&=\bar{\id\otimes\zeta_G}\circ\bar{k_X\otimes\id}\circ\sigma
\\&=\bar{\id\otimes\zeta_G}\circ\zeta\circ k_X,
\end{align*}
and similarly
\[
\bar{\zeta\otimes\id}\circ\zeta\circ k_A=\bar{\id\otimes\zeta_G}\circ\zeta\circ k_A,
\]
and it follows that
\[
\bar{\zeta\otimes\id}\circ\zeta=\bar{\id\otimes\zeta_G}\circ\zeta.
\]

For the coaction-nondegeneracy, routine computations show that
\[
\clspn\bigl\{\zeta_X(X)(1\otimes \cg)\bigr\}
=k_X(X)\otimes \cg,
\]
and of course
\[
\clspn\bigl\{\zeta_A(A)(1\otimes \cg)\bigr\}
=k_A(A)\otimes \cg,
\]
and hence the property \eqref{katayama} holds.
\end{proof}

We now develop a 
few tools involving inner coactions on correspondences, for use elsewhere.

\begin{prop}\label{inner}
Let $X$ be an $A-B$ correspondence,
and let $\mu:C_0(G)\to M(A)$ and $\nu:C_0(G)\to M(B)$ be nondegenerate homomorphisms,
and let $\delta^\mu$ and $\delta^\nu$ be the associated inner coactions on $A$ and $B$.
Then there is a $\delta^\mu-\delta^\nu$ compatible coaction $\sigma$ on $X$ given by
\[
\sigma(\xi)=\bar{\mu\otimes\id}(w_G)\cdot (\xi\otimes 1)\cdot \bar{\nu\otimes\id}(w_G)^*.
\]
\end{prop}

\begin{proof}
Write
\[
u=\bar{\mu\otimes\id}(w_G)\midtext{and}v=\bar{\nu\otimes\id}(w_G).
\]
Then $u\in M(A\otimes \cg)$, $v\in M(B\otimes \cg)$, and
\[
X\otimes 1\subset M(X\otimes \cg),
\]
so certainly $\sigma$ maps into $M(X\otimes \cg)$.

To see that $(\delta,\sigma,\epsilon)$ is a correspondence homomorphism, we compute, for $a\in A$ and $\xi,\eta\in X$:
\begin{align*}
\sigma(a\cdot \xi)
&=u\cdot (a\cdot \xi\otimes 1)\cdot v^*
\\&=u\cdot \bigl((a\otimes 1)\cdot (\xi\otimes 1)\bigr)\cdot v^*
\\&=u(a\otimes 1)u^*u\cdot (\xi\otimes 1)\cdot v^*
\\&=\delta(a)\cdot \sigma(\xi),
\end{align*}
and
\begin{align*}
\<\sigma(\xi),\sigma(\eta)\>
&=\<u\cdot (\xi\otimes 1)\cdot v^*,u\cdot (\eta\otimes 1)\cdot v^*\>
\\&=\<(\xi\otimes 1)\cdot v^*,(\eta\otimes 1)\cdot v^*\>
\\&\hspace{.5in}\text{(because $u$ is unitary)}
\\&=v\<\xi\otimes 1,\eta\otimes 1\>v^*
\\&=v\bigl(\<\xi,\eta\>\otimes 1\bigr)v^*
\\&=\epsilon(\<\xi,\eta\>).
\end{align*}

We show coaction-nondegeneracy:
\begin{align*}
&\clspn\{(1\otimes \cg)\cdot \sigma(X)\}
\\&\quad=\clspn\{(1\otimes \cg)u\cdot (X\otimes 1)\cdot v^*\}
\\&\quad=\clspn\{(1\otimes \cg)u\cdot (\mu(C_0(G)\cdot X\otimes 1)\cdot v^*\}
\\&\quad=\clspn\{(1\otimes \cg)u(\mu(C_0(G))\otimes 1)\cdot (X\otimes 1)\cdot v^*\}
\\&\quad=\clspn\{(1\otimes \cg)(\mu(C_0(G))\otimes 1)u\cdot (X\otimes 1)\cdot v^*\}
\\&\hspace{1in}\text{(because $u\in M(\mu(C_0(G))\otimes \cg)$)}
\\&\quad=\clspn\{(\mu(C_0(G))\otimes \cg)u\cdot (X\otimes 1)\cdot v^*\}
\\&\quad=\clspn\{(\mu(C_0(G))\otimes \cg)\cdot (X\otimes 1)\cdot v^*\}
\\&\hspace{1in}\text{(because $u$ is a unitary multiplier)}
\\&\quad=(X\otimes \cg)\cdot v^*
\\&\quad=(X\otimes \cg),
\end{align*}
because $v$ is unitary,
and similarly
\[
\clspn\{\sigma(X)\cdot (1\otimes \cg)\}=X\otimes \cg.
\]
This also implies that $\sigma$ is nondegenerate as a correspondence homomorphism.

For the coaction identity, we have
\begin{align*}
\bar{\sigma\otimes\id}\circ\sigma(\xi)
&=u_{12}\cdot \sigma(\xi)_{13}\cdot v_{12}^*
\\&=u_{12}u_{13}\cdot (\xi\otimes 1\otimes 1)\cdot v_{13}^*v_{12}^*
\\&=\bar{\id\otimes\delta_G}(u)\cdot \bar{\id\otimes\delta_G}(\xi\otimes 1)\cdot \bar{\id\otimes\delta_G}(v)^*
\\&=\bar{\id\otimes\delta_G}\circ\sigma(\xi),
\end{align*}
where the third equality expresses the fact that $u$ and $v$ are ``corepresentations'' of $C_0(G)$,
and where the first equality follows from linearity, density, strict continuity, and the following computation with an elementary tensor $\eta\otimes c\in X\odot \cg$:
\begin{align*}
\bar{\sigma\otimes\id}(\eta\otimes c)
&=\sigma(\eta)\otimes c
\\&=u\cdot (\eta\otimes 1)\cdot v^*\otimes c
\\&=(u\otimes 1)\cdot (\eta\otimes 1\otimes c)\cdot (v\otimes 1)^*
\\&=u_{12}\cdot (\eta\otimes c)_{13}\cdot v_{12}^*.
\qedhere
\end{align*}
\end{proof}

\begin{defn}
In the situation of \propref{inner}, we call the coaction $\sigma$ on $X$ \emph{inner}, and say that it is \emph{implemented} by the pair $(\mu,\nu)$.
\end{defn}

\begin{cor}\label{unitary}
Let $(X,A)$ be a correspondence,
let $\delta$ be a coaction of $G$ on $A$,
and let $\mu:C_0(G)\to \LL(X)$ be a nondegenerate representation such that the pair $(\varphi_A,\mu)$ is a covariant representation of the coaction $(A,\delta)$.
Define a unitary
\[
u=\bar{\mu\otimes\id}(w_G)\in \LL(X\otimes \cg).
\]
Then there is an $\delta-\delta^1$ compatible coaction $\sigma$ on $X$ given by
\[
\sigma(\xi)=u\cdot (\xi\otimes 1).
\]
\end{cor}

\begin{proof}
Temporarily regard $X$ as a $\KK(X)-A$ correspondence.
Letting $\delta^\mu$ be the inner coaction on $\KK(X)$ implemented by $\mu$,
by \propref{inner} the formula for $\sigma$ defines a $\delta^\mu-\delta^1$ compatible coaction on $X$.
Since $(\varphi_A,\mu)$ is covariant for $(A,\delta)$, it follows that $\sigma$ is also $\delta-\delta^1$ compatible.
\end{proof}

\begin{cor}\label{commute}
Let $(X,A)$ be a correspondence,
and let $\mu:C_0(G)\to\LL(X)$ be a nondegenerate representation commuting with $\varphi_A$.
Then there is a coaction $\zeta$ of $G$ on $\OO_X$ such that for $\xi\in X$ and $a\in A$ we have
\begin{align*}
\zeta\circ k_X(\xi)&=\bar{k_X\otimes\id}\Bigl(\bar{\mu\otimes 1}(w_G)\cdot (\xi\otimes 1)\Bigr)\\
\zeta\circ k_A(a)&=k_A(a)\otimes 1.
\end{align*}
\end{cor}

\begin{proof}
Since $\mu$ commutes with $\varphi_A$,
the hypotheses of \corref{unitary} are satisfied when $\delta$ is taken to be the trivial coaction $\delta^1$, and we let $\sigma$ be the resulting $\delta^1-\delta^1$ compatible coaction on $X$.
Then \propref{coaction} gives a suitable coaction $\zeta$ of $G$ on $\OO_X$,
because
the trivial coaction $\delta^1$ 
maps $J_X$ into
\[
J_X\otimes 1\subset M(A\otimes \cg;J_X\otimes \cg).
\qedhere
\]
\end{proof}

\section{Crossed products}\label{crossed products}

\begin{lem}\label{covariant}
Let $(\sigma,\delta)$ be a coaction of $G$ on a correspondence $(X,A)$
such that $\delta(J_X)\subset M(A\otimes C^*(G);J_X\otimes C^*(G))$, and
let $(\psi,\pi,\mu):(X,A,C_0(G))\to M(B)$ be a $(\sigma,\delta)$-covariant homomorphism, with $(\psi,\pi)$ Cuntz-Pimsner covariant.
Then the pair
\[
(\psi\times\pi,\mu):(\OO_X,C_0(G))\to M(B)
\]
is covariant for the associated coaction $\zeta$ of $G$ on $\OO_X$.
\end{lem}

\begin{proof}
$\pi$ and $\mu$ are nondegenerate, hence so is $\psi\times\pi$.
Let $u=\bar{\mu\otimes\id}(w_G)$. We must show that for $x\in \OO_X$ we have
\[
\bar{(\psi\times\pi)\otimes\id}\circ\zeta(x)
=\ad u\bigl((\psi\times\pi)(x)\otimes 1),
\]
and it suffices to show this 
on generators
$k_X(\xi)$ and $k_A(a)$ for $\xi\in X$ and $a\in A$.
For for the first, we have
\begin{align*}
\bar{(\psi\times\pi)\otimes\id}\circ\zeta\circ k_X(\xi)
&=\bar{(\psi\times\pi)\otimes\id}\circ\bar{k_X\otimes\id}\circ\sigma(\xi)
\\&=\bar{\psi\otimes\id}\circ\sigma(\xi)
\\&=u \bigl(\psi(\xi)\otimes 1\bigr)u^*
\\&=\ad u\bigl((\psi\times\pi)\circ k_X(\xi)\otimes 1\bigr),
\end{align*}
and for the second,
\begin{align*}
\bar{(\psi\times\pi)\otimes\id}\circ\zeta\circ k_A(a)
&=\bar{(\psi\times\pi)\otimes\id}\circ\bar{k_A\otimes\id}\circ\delta(a)
\\&=\bar{\pi\otimes\id}\circ\delta(a)
\\&=\ad u\bigl(\pi(a)\otimes 1\bigr)
\\&=\ad u\bigl((\psi\times\pi)\circ k_A(a)\otimes 1\bigr).
\qedhere
\end{align*}
\end{proof}

\begin{lem}\label{composition}
Let $(\sigma,\delta)$ be a coaction of $G$ on a correspondence $(X,A)$,
let $(\psi,\pi,\mu):(X,A,C_0(G))\to (M_B(Y),M(B))$ be a $(\sigma,\delta)$-covariant correspondence homomorphism, and
let $(\rho,\tau):(Y,B)\to (M_D(Z),M(D))$ be a correspondence homomorphism with $\tau$ nondegenerate.
Then the composition
\[
\bigl(\bar\rho\circ\psi,\bar\tau\circ\pi,\bar\tau\circ\mu\bigr):(X,A,C_0(G))\to (M_D(Z),M(D))
\]
is covariant for $(\sigma,\delta)$.
\end{lem}

\begin{proof}
First of all, since $\pi$, $\mu$, and $\tau$ are nondegenerate, $\bar\tau\circ\pi$ is also nondegenerate, and $(\bar\tau\circ\pi,\bar\tau\circ\mu)$ is covariant for $(A,\delta)$ by the standard theory of $C^*$-coactions.

Routine calculations show that
\[
(\bar\rho\circ\psi,\bar\tau\circ\pi):(X,A)\to (M(Z),M(D))
\]
is a correspondence homomorphism.
Also, since $\psi$ and $\rho$ map into $M_B(Y)$ and $M_D(Z)$, respectively, it is easy to see that $\bar\rho\circ\psi$ maps $X$ into $M_D(Z)$.

Letting $u=\bar{\bar\tau\circ\mu\otimes\id}(w_G)$, 
the following calculation completes the proof:
for $\xi\in X$ we have
\begin{align*}
\bar{(\bar\tau\circ\psi)\otimes\id}\circ\sigma(\xi)
&=\bar{\tau\otimes\id}\circ\bar{\psi\otimes\id}\circ\sigma(\xi)
\\&=\bar{\tau\otimes\id}\Bigl(\bar{\mu\otimes\id}(w_G)\cdot \bigl(\psi(\xi)\otimes 1\bigr)\cdot\bar{\mu\otimes\id}(w_G)^*\Bigr)
\\&=u\cdot \bigl(\bar\tau\circ\psi(\xi)\otimes 1\bigr)\cdot u^*.
\qedhere
\end{align*}
\end{proof}

\begin{cor}\label{covariant CP}
Let $(\sigma,\delta)$ be a coaction of $G$ on a correspondence $(X,A)$
such that $\delta(J_X)\subset M(A\otimes C^*(G);J_X\otimes C^*(G))$, and
let $(\psi,\pi,\mu):(X,A,C_0(G))\to (M_B(Y),M(B))$ be a $(\sigma,\delta)$-covariant homomorphism, with $(\psi,\pi)$ Cuntz-Pimsner covariant.
Then the pair
\[
(\OO_{\psi,\pi},\bar{k_B}\circ\mu):(\OO_X,C_0(G))\to M(\OO_Y)
\]
is covariant for the associated coaction $\zeta$.
\end{cor}

\begin{proof}
Applying \lemref{composition} to the Toeplitz representation $(k_Y,k_B):(Y,B)\to \OO_Y$, we see that
\[
(\bar{k_Y}\circ\psi,\bar{k_B}\circ\pi,\bar{k_B}\circ\mu):(X,A,C_0(G))\to M(\OO_Y)
\]
is covariant for $(\sigma,\delta)$.

By 
\cite[Theorem~3.5]{KQRCorrespondenceFunctor}
the composition 
$(\bar{k_Y}\circ\psi,\bar{k_B}\circ\pi)$ is a Cuntz-Pimsner-covariant Toeplitz representation of $(X,A)$ in $M(\OO_Y)$.
Then, since $(\psi,\pi)$ is Cuntz-Pimsner covariant, \lemref{covariant} with $B=\OO_Y$ tells us that
\[
\bigl((\bar{k_Y}\circ\psi)\times (\bar{k_B}\circ\pi),\bar{k_B}\circ\mu\bigr):
(\OO_X,C_0(G))\to M(\OO_Y)
\]
is $\zeta$-covariant.
But by construction (see \cite[Corollary~3.6]{KQRCorrespondenceFunctor}) we have
\[
(\bar{k_Y}\circ\psi)\times (\bar{k_B}\circ\pi)=\OO_{\psi,\pi}.
\qedhere
\]
\end{proof}

\begin{thm}\label{crossed}
Let $(\sigma,\delta)$ be a
coaction of $G$ on a correspondence $(X,A)$
such that $\delta(J_X)\subset M(A\otimes C^*(G);J_X\otimes C^*(G))$, and let $\zeta$ be the associated coaction on $\OO_X$, as in \propref{coaction}.
If the canonical correspondence homomorphism
\[
(j_X,j_A):(X,A)\to (M(X\rtimes_\sigma G),M(A\rtimes_\delta G))
\]
is Cuntz-Pimsner covariant,
then
\[
 \OO_X\rtimes_\zeta G\cong\OO_{X\rtimes_\sigma G}.
\]
\end{thm}

\begin{rem}
We do not know whether the hypothesis of Cuntz-Pimsner covariance of $(j_X,j_A)$ is redundant; in \corref{implies CP} below we will show that it is satisfied under certain conditions.
\end{rem}

\begin{proof}[Proof of \thmref{crossed}]
Our strategy is to construct a covariant homomorphism
\[
(\rho,\mu):(\OO_X,C_0(G))\to M(\OO_{X\rtimes_\sigma G}),
\]
and show that the integrated form $\rho\times\mu$ is an isomorphism of $\OO_X\rtimes_\zeta G$ onto $\OO_{X\rtimes_\sigma G}$.
For the covariant homomorphism we will need a homomorphism of $\OO_X$, and to get this we will apply functoriality:
since $(j_X,j_A)$ is Cuntz-Pimsner covariant, by \cite[Corollary~3.6]{KQRCorrespondenceFunctor}
there is a unique nondegenerate homomorphism
\[
\OO_{j_X,j_A}:\OO_X\to M(\OO_{X\rtimes_\sigma G})
\]
making the diagram
\[
\xymatrix@C+30pt{
(X,A) \ar[r]^-{(j_X,j_A)} \ar[d]_{(k_A,k_A)}
&(M_{A\rtimes_\delta G}(X\rtimes_\sigma G),M(A\rtimes_\delta G))
\ar[d]^{(\bar{k_{X\rtimes_\sigma G}},\bar{k_{A\rtimes_\delta G}})}
\\\
\OO_X \ar[r]_{\OO_{j_X,j_A}}
&M(\OO_{X\rtimes_\sigma G})
}
\]
commute.

We next show that $(j_X,j_A,j_G)$ is covariant for $(\sigma,\delta)$:
\begin{align*}
\bar{j_X\otimes\id}\circ\sigma
&=\bar{\bigl(\bar{\id\otimes\lambda}\circ\sigma\bigr)\otimes\id}\circ\sigma
\\&=\bar{\id\otimes\lambda\otimes\id}\circ\bar{\sigma\otimes\id}\circ\sigma
\\&=\bar{\id\otimes\lambda\otimes\id}\circ\bar{\id\otimes\delta_G}\circ\sigma
\\&=\ad\bigl(1\otimes\bar{M\otimes\id}(w_G\bigr)\circ\bar{\id\otimes\lambda\otimes\id}\circ(\sigma\otimes 1)
\\&=\ad\bar{\bar{1\otimes M}\otimes\id}(w_G)\circ\bigl(\bar{\id\otimes\lambda}\circ\sigma\otimes 1\bigr)
\\&=\ad\bar{j_G\otimes\id}(w_G)\circ(j_X\otimes 1),
\end{align*}
where the fourth equality follows 
by linearity, density, and strict continuity from the following computation with elementary tensors: for $\eta\in X$ and $t\in G$ we have
\begin{align*}
\bar{\id\otimes\lambda\otimes\id}\circ\bar{\id\otimes\delta_G}(\eta\otimes t)
&=\bar{\id\otimes\lambda\otimes\id}\bigl(\eta\otimes\delta_G(t)\bigr)
\\&=\eta\otimes\bar{\lambda\otimes\id}\circ\delta_G(t)
\\&=\eta\otimes\lambda_t\otimes t
\\&=\eta\otimes\ad\bar{M\otimes\id}(w_G)(\lambda_t\otimes 1)
\\&=\ad\bigl(1\otimes\bar{M\otimes\id}(w_G)\bigr)(\eta\otimes\lambda_t\otimes 1)
\end{align*}
where in turn the fourth equality follows from the following: for $f\in B(G)$ we have
\begin{align*}
S_f\bigl((\lambda_t\otimes t)\bar{M\otimes\id}(w_G)\bigr)
&=\lambda_tS_{f\cdot t}\bigl(\bar{M\otimes\id}(w_G)\bigr)
\\&=\lambda_tM_{f\cdot t}
\\&=M_f\lambda_t
\\&=S_f\bigl(\bar{M\otimes\id}(w_G)\bigr)\lambda_t
\\&=S_f\bigl(\bar{M\otimes\id}(w_G)(\lambda_t\otimes 1)\bigr),
\end{align*}
so that
\[
(\lambda_t\otimes t)\bar{M\otimes\id}(w_G)=\bar{M\otimes\id}(w_G)(\lambda_t\otimes 1).
\]

It now follows from \corref{covariant CP} that the pair
\[
\bigl(\OO_{j_X,j_A},\bar{k_{A\rtimes_\delta G}}\circ j_G\bigr)
\]
is a covariant homomorphism of the coaction $(\OO_X,\zeta)$ in $M(\OO_{X\rtimes_\sigma G})$,
and thus we get a homomorphism
\[
\Pi:=\OO_{j_X,j_A}\times \bigl(\bar{k_{A\rtimes_\delta G}}\circ j_G\bigr):
\OO_X\rtimes_\zeta G\to M(\OO_{X\rtimes_\sigma G}).
\]

It remains to show the following:
\begin{enumerate}
\item $\Pi$ maps into $\OO_{X\rtimes_\sigma G}$;

\item $\Pi$ is surjective;

\item $\Pi$ is injective.
\end{enumerate}

For (i), for $\xi\in X$, $a\in A$, and $f\in C_0(G)$ we have
\begin{align*}
&\OO_{j_X,j_A}\circ k_X(\xi)\bar{k_{A\rtimes_\delta G}}\circ j_G(f)
\\&\quad=\bar{k_{X\rtimes_\sigma G}}(j_X(\xi))\bar{k_{A\rtimes_\delta G}}(j_G(f))
\\&\quad=\bar{k_{X\rtimes_\sigma G}}\bigl(j_X(\xi)\cdot j_G(f))\bigr)
\end{align*}
and
\begin{align*}
&\OO_{j_X,j_A}\circ k_A(a)\bar{k_{A\rtimes_\delta G}}\circ j_G(f)
\\&\quad=\bar{k_{A\rtimes_\delta G}}(j_A(a))\bar{k_{A\rtimes_\delta G}}(j_G(f))
\\&\quad=\bar{k_{A\rtimes_\delta G}}\bigl(j_A(a)j_G(f))\bigr).
\end{align*}

For (ii), we see from the above that the image of $\Pi$ contains
\[
\bar{k_{X\rtimes_\sigma G}}\bigl(j_X(X)\cdot j_G(C_0(G))\bigr)
\midtext{and}
\bar{k_{A\rtimes_\delta G}}\bigl(j_A(A)\cdot j_G(C_0(G))\bigr),
\]
and hence contains
\[
\bar{k_{X\rtimes_\sigma G}}(X\rtimes_\sigma G)
\midtext{and}
\bar{k_{A\rtimes_\delta G}}(A\rtimes_\delta G),
\]
which generate $\OO_{X\rtimes_\sigma G}$.

For (iii) we apply \cite[Theorem~3.1]{QuiggLandstadDuality}:
we must show that $\Pi\circ j_{\OO_X}$ is faithful and that there is an action $\alpha$ of $G$ on $\OO_{X\rtimes_\sigma G}$ such that $\Pi$ is $\what\zeta-\alpha$ equivariant.

To see that $\Pi\circ j_{\OO_X}$ is faithful, we apply the Gauge-Invariant Uniqueness Theorem: since
\[
\Pi\circ j_{\OO_X}\circ k_A
=\OO_{j_X,j_X}\circ k_A
=j_A
\]
is faithful, it suffices to show that for all $z\in\T$, $\xi\in X$, and $a\in A$ we have
\begin{align*}
\gamma_z\circ\Pi\circ j_{\OO_X}\circ k_X(\xi)
&=z\Pi\circ j_{\OO_X}\circ k_X(\xi)
\\
\gamma_z\circ\Pi\circ j_{\OO_X}\circ k_A(a)
&=\Pi\circ j_{\OO_X}\circ k_A(a).
\end{align*}
For the first, we have
\begin{align*}
\gamma_z\circ\Pi\circ j_{\OO_X}\circ k_X(\xi)
&=\gamma_z\circ\OO_{j_X,j_A}\circ k_X(\xi)
\\&=\gamma_z\circ\bar{k_{X\rtimes_\sigma G}}\circ j_X(\xi)
\\&=z\bar{k_{X\rtimes_\sigma G}}\circ j_X(\xi)
\\&=z\Pi\circ j_{\OO_X}\circ k_X(\xi),
\end{align*}
where the third equality follows from
\[
\gamma_z\circ k_{X\rtimes_\sigma G}=zk_{X\rtimes_\sigma G}.
\]
The second is similar, this time using $\gamma_z\circ k_{A\rtimes_\delta G}=k_{A\rtimes_\delta G}$.

We now turn to the action of $G$.
First note that there is an action $\beta$ of $G$ on $X\rtimes_\sigma G$ given by
\[
\beta_t\bigl(j_X(\xi)\cdot j_G(f)\bigr)=j_X(\xi)\cdot j_G\circ \rt_t(f)
\midtext{for}\xi\in X,f\in C_0(G),
\]
where $\rt$ is the action of $G$ on $C_0(G)$ given by right translation.
This in turn gives an action $\alpha$ of $G$ on $\OO_{X\rtimes_\sigma G}$ such that
\begin{align*}
\alpha_t\circ k_{X\rtimes_\sigma G}&=k_{X\rtimes_\sigma G}\circ\beta_t\\
\alpha_t\circ k_{A\rtimes_\delta G}&=k_{A\rtimes_\delta G}\circ\beta_t.
\end{align*}
Finally, we check the $\what\zeta-\alpha$ covariance:
\begin{align*}
\alpha_t\circ\Pi\circ j_{\OO_X}
&=\alpha_t\circ\OO_{j_X,j_A}
\\&=\alpha_t\circ\bar{k_{X\rtimes_\sigma G}}\circ j_X
\\&=\bar{k_{X\rtimes_\sigma G}}\circ\beta_t\circ j_X
\\&=\bar{k_{X\rtimes_\sigma G}}\circ j_X
\\&=\Pi\circ j_{\OO_X}
\\*&=\Pi\circ\what\zeta_t\circ j_{\OO_X},
\end{align*}
and
\begin{align*}
\alpha_t\circ\Pi\circ j_G
&=\alpha_t\circ\bar{k_{A\rtimes_\delta G}}\circ j_G
\\&=\bar{k_{A\rtimes_\delta G}}\circ\beta_t\circ j_G
\\&=\bar{k_{A\rtimes_\delta G}}\circ j_G\circ\rt_t
\\&=\Pi\circ j_G\circ\rt_t
\\&=\Pi\circ\what\zeta_t\circ j_G.
\qedhere
\end{align*}
\end{proof}

\begin{cor}\label{implies CP}
Let $(\sigma,\delta)$ be a 
coaction of $G$ on a correspondence $(X,A)$.
If either
\begin{enumerate}
\item $G$ is amenable, or

\item $\varphi_A:A\to \LL(X)$ is faithful,
\end{enumerate}
then the canonical correspondence homomorphism
\[
(j_X,j_A):(X,A)\to (M(X\rtimes_\sigma G),M(A\rtimes_\delta G))
\] is Cuntz-Pimsner covariant.
\end{cor}

\begin{proof}
By \lemref{j CP}, it suffices to  show that
\[
j_A(J_X)(A\rtimes_\delta G)\subset J_{X\rtimes_\sigma G}.
\]
The ideal of $A\rtimes_\delta G$ generated by $j_A(J_X)(A\rtimes_\delta G)$ is
\[
I:=\clspn\{(A\rtimes_\delta G)j_A(J_X)(A\rtimes_\delta G)\},
\]
so it suffices to show that $\varphi_{A\rtimes_\delta G}$ maps $I$ injectively into $\KK(X\rtimes_\sigma G)$.
As we observed immediately before \remref{j nondegenerate}, we can
work with $\varphi_A\rtimes G$ and $\KK(X)\rtimes_{\sigma^{(1)}} G$ rather than $\varphi_{A\rtimes_\delta G}$ and $\KK(X\rtimes_\sigma G)$.
To see that $\varphi_A\rtimes G$ maps $I$ into $\KK(X)\rtimes_{\sigma^{(1)}} G$, it suffices to observe that
\begin{align*}
&(\varphi_A\rtimes G)\bigl(j_G^A(C_0(G))j_A(A)j_A(J_X)j_A(A)j_G^A(C_0(G))\bigr)
\\&\quad=(\varphi_A\rtimes G)\bigl(j_G^A(C_0(G))j_A(AJ_XA)j_G^A(C_0(G))\bigr)
\\&\quad\subset(\varphi_A\rtimes G)\bigl(j_G^A(C_0(G))j_A(J_X)j_G^A(C_0(G))\bigr)
\\&\quad=j_G^{\KK(X)}(C_0(G))\bar{j_{\KK(X)}}(\varphi_A(J_X))j_G^{\KK(X)}(C_0(G))
\\&\quad\subset j_G^{\KK(X)}(C_0(G))j_{\KK(X)}(\KK(X))j_G^{\KK(X)}(C_0(G))
\\&\quad\subset \KK(X)\rtimes_{\sigma^{(1)}} G.
\end{align*}

On the other hand, to see that $\varphi_A\rtimes G$ is injective on $I$,
we now 
consider each hypothesis (i) and (ii) separately.
First, if $\varphi_A$ is injective, then so is $\varphi_A\rtimes G$,
because $\varphi_A$ gives a $G$-equivariant isomorphism between $(A,\delta)$ and the image $(\varphi_A(A),\eta)$, where $\eta$ is the corresponding coaction on $\varphi_A(A)$, and we have a commuting diagram
\[
\xymatrix{
A\rtimes_\delta G \ar[r]^-\cong \ar[dr]_(.4){\varphi_A\rtimes G}
&\varphi_A(A)\rtimes_\eta G \ar@{^(->}[d]
\\
&M(\KK(X)\rtimes_{\sigma^{(1)}} G),
}
\]
where the horizontal arrow is an isomorphism and the vertical arrow is an inclusion.

Thus it remains to show that  $\varphi_A\rtimes G$ is injective on $I$ under the
assumption that $G$ is amenable.
We will show that in this case $J_X$ is a $\delta$-invariant ideal of $A$ in the sense that $\delta$ restricts to a coaction on $J_X$.
It will follow that
\[
I=J_X\rtimes_\delta G,
\]
and since the restriction $\varphi_A|:J_X\to \KK(X)$ is injective 
we will be able to conclude that
\[
\varphi_A|\rtimes G:J_X\rtimes_\delta G\to \KK(X)\rtimes_{\sigma^{(1)}} G
\]
is injective as well.

To see that $J_X$ is invariant, by \cite[Proposition~2.6]{QuiggLandstadDuality} it suffices to show that $J_X$ is an $A(G)$-submodule of $A$.
Let $f\in A(G)$ and $a\in J_X$. We must show both of the following:
\begin{enumerate}
\item $\varphi_A(f\cdot a)\in \KK(X)$;

\item $(f\cdot a)b=0$ for all $b\in \ker\varphi_A$.
\end{enumerate}
For (i), we have
\begin{align*}
\varphi_A(f\cdot a)
&=\varphi_A\circ S_f\circ\delta(a)
\\&=S_f\circ\bar{\varphi_A\otimes\id}\circ\delta(a)
\\&=S_f\circ\bar{\sigma^{(1)}}\circ\varphi_A(a)
\\&\subset S_f\circ\sigma^{(1)}(\KK(X))
\\&\subset S_f\Bigl(M_{\cg}\bigl(\KK(X)\otimes \cg\bigr)\Bigr)
\\&\subset \KK(X),
\end{align*}
by \cite[Lemma~1.5]{lprs}.

In preparation for (ii), we first show that $\ker\varphi_A$ is $\delta$-invariant: if $f\in A(G)$ and $b\in \ker\varphi_A$, then
\begin{align*}
\varphi_A(f\cdot b)
&=S_f\circ\bar{\varphi_A\otimes\id}\circ\delta(b)
=S_f\circ\bar{\sigma^{(1)}}\circ\varphi_A(b)
=0.
\end{align*}
Thus $\delta$ restricts to a coaction on $\ker\varphi_A$, so
\begin{equation}\label{kernel nondegenerate}
\clspn\{\delta(\ker\varphi_A)(1\otimes \cg)\}=\ker\varphi_A\otimes \cg.
\end{equation}
We now verify (ii): for $f\in A(G)$, $a\in J_X$, and $b\in \ker\varphi_A$ we first factor $f=c\cdot f'$ for some $c\in \cg$ and $f'\in A(G)$ (using amenability of $G$ again), and then
\begin{align*}
(f\cdot a)b
&=S_f\circ\delta(a)b
\\&=S_f\bigl(\delta(a)(b\otimes 1)\bigr)
\\&=S_{c\cdot f'}\bigl(\delta(a)(b\otimes 1)\bigr)
\\&=S_{f'}\bigl(\delta(a)(b\otimes c)\bigr)
\\&\approx \sum_1^nS_{f'}\bigl(\delta(a)\delta(b_i)(1\otimes c_i)\bigr)
\\&\hspace{.5in}\text{(for some $b_i\in \ker\varphi_A$ and $c_i\in C$, by \eqref{kernel nondegenerate})}
\\&\approx \sum_1^nS_{f'}\bigl(\delta(ab_i)(1\otimes c_i)\bigr)
\\&=0,
\end{align*}
because $J_X\subset (\ker\varphi_A)^\perp$.
\end{proof}

Then \cite[Theorem~6.9]{QuiggDualityTwisted} (see also \cite[Theorem~2.9]{lprs})
shows that the crossed product of a $C^*$-algebra $A$ by an inner coaction of $G$ is isomorphic to $A\otimes C_0(G)$;
the following result is a version for  correspondences:

\begin{prop}\label{inner crossed}
Let $(A,X,B)$ be a correspondence, and let $(A,\delta)$ and $(B,\epsilon)$ be inner coactions implemented by nondegenerate homomorphisms $\mu$ and $\nu$, respectively, and let $\sigma$ be the associated coaction on $X$, as in \propref{inner}. Then
there is an isomorphism
\[
\Phi:X\rtimes_\sigma G\to X\otimes C_0(G)
\]
given by
\[
\Phi(y)=\bar{\mu\otimes\lambda}(w_G)^*\cdot y \cdot \bar{\nu\otimes\lambda}(w_G).
\]
The left and right module actions are transformed  by $\Phi$ as follows:
\begin{align*}
&\Phi\bigl(j_A(a)j_G^A(f)\cdot y\cdot j_B(b)j_G^B(g)\bigr)
\\&\quad=(a\otimes 1)(\mu\times M)(f)\cdot \Phi(y)\cdot (b\otimes 1)(\nu\times M)(g),
\end{align*}
where $\mu\times M$ denotes the Kronecker product of $\mu$ and $M$, respectively, and similarly for $\nu\times M$.
\end{prop}

\begin{proof}
Note that we are identifying $C_0(G)$ with its image under the representation $M$ on $L^2(G)$ by pointwise multiplication, i.e., $(M_f\xi)(t)=f(t)\xi(t)$ for $f\in C_0(G)$ and $\xi\in L^2(G)$.
Routine calculations show
\begin{align*}
\Phi\circ j_A&=\id_A\otimes 1\\
\Phi\circ j_B&=\id_B\otimes 1\\
\Phi\circ j_X&=\id_X\otimes 1\\
\Phi\circ j_G^A&=\mu\times M\\
\Phi\circ j_G^B&=\nu\times M;
\end{align*}
for the last two it helps to note that
\[
\ad\bar{\id\otimes\lambda}(w_G)^*(1\otimes M_f)=(\id\times M)(f).
\]
Since
\begin{align*}
\bar{\mu\otimes\lambda}(w_G)&\in M(A\otimes \KK(L^2(G))),
\end{align*}
and similarly for $\bar{\nu\otimes\lambda}(w_G)$, clearly $\Phi$ maps $X\rtimes_\sigma G$ into $M(X\otimes L^2(G))$.
We actually have $\Phi(X\rtimes_\sigma G)=X\otimes C_0(G)$, because
\begin{align*}
&\clspn\bigl\{\Phi\bigl(j_X(X)\cdot j_G(C_0(G))\bigr\}
\\&\quad=\clspn\bigl\{\Phi\bigl(j_X(X\cdot B)\cdot j_G(C_0(G))\bigr\}
\\&\quad=\clspn\bigl\{\Phi\bigl(j_X(X)\cdot j_B(B)j_G(C_0(G))\bigr\}
\\&\quad=\clspn\bigl\{(X\otimes 1)\cdot \ad\bar{\nu\otimes\id}(w_G)^*(B\rtimes_\epsilon G)\bigr\}
\\&\quad=\clspn\bigl\{(X\otimes 1)\cdot (B\otimes C_0(G))\bigr\}
\\&\hspace{1in}\text{(by \cite[Theorem~6.9]{QuiggDualityTwisted} or \cite[Theorem~2.9]{lprs})}
\\&\quad=X\otimes C_0(G).
\qedhere
\end{align*}
\end{proof}

Let $(\gamma,\alpha)$ be an action of $G$ on a correspondence $(X,A)$.
Assume that $G$ is amenable; in particular, there is no difference between the full and reduced crossed products $X\rtimes_\gamma G$ and $X\rtimes_{\gamma,r} G$ (and similarly for $A$), so we can freely apply the results of \cite[Section 3.1]{enchilada}.

As in \cite[Proposition~3.5]{enchilada}, let $\what\gamma$ be the dual coaction of $G$ on $X\rtimes_\gamma G$, determined on generators $\xi\in C_c(G,X)$ by
\[
\what\gamma(\xi)(t)=\xi(t)\otimes t,
\]
so that $\what\gamma$ is an element of $C_c(G,M^\beta(X\otimes \cg))$,
which in turn is embedded in $M((X\rtimes_\gamma G)\otimes \cg)$ via the isomorphism \cite[Lemma~3.4]{enchilada}
\[
(X\rtimes_\gamma G)\otimes \cg\xrightarrow{\cong}
(X\otimes \cg)\rtimes_{\gamma\otimes\id} G
\]
that extends the canonical embedding
\[
C_c(G,X)\odot \cg\hookrightarrow C_c(G,X\otimes \cg).
\]

\begin{prop}\label{dual coaction}
Let $(\gamma,\alpha)$ be an action of $G$ on a correspondence $(X,A)$, and assume that $G$ is amenable.
Then the dual coaction $(\what\gamma,\what\alpha)$ on $(X\rtimes_\gamma G,A\rtimes_\alpha G)$ 
satisfies
\begin{equation}\label{dual}
\what\alpha(J_{X\rtimes_\gamma G})\subset
M\bigl((A\rtimes_\alpha G)\otimes C^*(G);J_{X\rtimes_\gamma G}\otimes C^*(G)\bigr).
\end{equation}
\end{prop}

\begin{proof}
By \cite[Proposition~2.7]{HN}, the ideal $J_X$ of $A$ is $\alpha$-invariant, and
\[
J_{X\rtimes_\gamma G}=J_X\rtimes_\alpha G.
\]
The isomorphism
\[
(A\rtimes_\alpha G)\otimes \cg\xrightarrow{\cong} (A\otimes \cg)\rtimes_{\alpha\otimes\id} G,
\]
of 
\cite[Lemma~A.20]{enchilada} clearly takes $(J_X\rtimes_\alpha G)\otimes \cg$ to $(J_X\otimes \cg)\rtimes_{\alpha\otimes\id} G$.

Recall that $\what\alpha$ takes a function $f\in C_c(G,A)$ to the function in $C_c(G,M^\beta(A\otimes \cg))$ defined by
\[
\what\alpha(f)(t)=f(t)\otimes t.
\]
It follows that for $g\in C_c(G,A\otimes \cg)$ we have
\begin{align*}
\bigl(\what\alpha(f)g\bigr)(t)
&=\int_G\what\alpha(f)(s)\bar{\alpha_s\otimes\id}(g(s\inv t))\,ds
\\&=\int_G\bigl(f(s)\otimes s\bigr)\bar{\alpha_s\otimes\id}(g(s\inv t))\,ds.
\end{align*}
Now let $f\in C_c(G,J_X)$.
For all $s\in G$, it is easy to check, by first computing with elementary tensors $a\otimes c\in A\odot \cg$, that
\[
\bigl(f(s)\otimes s\bigr)(A\otimes \cg)\subset J_X\otimes \cg,
\]
and it follows that
\[
\what\alpha(f)g\in C_c(G,J_X\otimes \cg)
\subset (J_X\otimes \cg)\rtimes_{\alpha\otimes\id} G.
\]
By density, this implies that
\[
\what\alpha\bigl(J_X\rtimes_\alpha G\bigr)\subset
M\bigl((A\otimes \cg)\rtimes_{\alpha\otimes\id} G;(J_X\otimes \cg)\rtimes_{\alpha\otimes\id} G\bigr),
\]
which in turn implies \eqref{dual}.
\end{proof}

\section{Application}\label{apps}

As an application of our techniques, we will give an alternative approach to 
a recent result of Hao and Ng
\cite[Theorem~2.10]{HN}.
Given an action $(\gamma,\alpha)$ of an amenable locally compact group $G$ on a nondegenerate correspondence $(X,A)$, 
Hao and Ng construct an isomorphism
\[
\OO_{X\rtimes_\gamma G}\iso \OO_X\rtimes_\beta G,
\]
where $X\rtimes_\gamma G$ is the crossed-product correspondence over $A\rtimes_\alpha G$
and $\beta$ is the associated action of $G$ on $\OO_X$.
In our earlier paper \cite[Proposition~4.3]{KQRCorrespondenceFunctor} we suggested an alternative approach to this result, removing the amenability hypothesis on $G$.
Namely, we construct a surjection that goes in the opposite direction:
\[
\OO_X\rtimes_\beta G\to \OO_{X\rtimes_\gamma G}.
\]
We suspect, but were unable to prove, that this is an isomorphism in general;
however, at least in the amenable case, we can give a new proof of \cite[Theorem~2.10]{HN} with the help of Propositions~\ref{dual coaction} and \ref{coaction}.

\begin{thm}\label{injective}
Let $(\gamma,\alpha)$ be an action of $G$ on a nondegenerate correspondence $(X,A)$,
let $\beta$ be the associated action of $G$ on $\OO_X$,
and let
\[
\Pi:=\OO_{i_X,i_A}\times u:\OO_X\rtimes_\beta G\to \OO_{X\rtimes_\gamma G}
\]
be the surjection from \cite[Proposition~4.3]{KQRCorrespondenceFunctor}.
If $G$ is amenable, then $\Pi$ is an isomorphism.
\end{thm}

\begin{proof}
By Propositions~\ref{dual coaction} and \ref{coaction}
we get a coaction $\zeta$ of $G$ on $\OO_{X\rtimes_\gamma G}$.
Our strategy is to show that $\Pi$ is $\what\beta-\zeta$ equivariant
and that $\OO_{i_X,i_A}$ is injective,
and then \cite[Proposition~3.1]{QuiggLandstadDuality} will imply that $\Pi$ is injective, because by amenability of $G$ the coaction $\zeta$ is automatically normal.

We check the equivariance condition
\[
\zeta\circ\Pi=\bar{\Pi\otimes\id}\circ\what\beta
\]
separately on generators from $X$, $A$, and $G$: for $X$ we have
\begin{align*}
\bar{\zeta\circ\Pi}\circ i_{\OO_X}\circ k_X
&=\bar\zeta\circ\bar\Pi\circ i_{\OO_X}\circ k_X
\\&=\bar\zeta\circ\OO_{i_X,k_A}\circ k_X
\\&=\bar\zeta\circ\bar{k_{X\rtimes_\gamma G}}\circ i_X
\\&=\bar{k_{X\rtimes_\gamma G}\otimes\id}\circ\bar{\what\gamma}\circ i_X
\\&=\bar{k_{X\rtimes_\gamma G}\otimes\id}\circ(i_X\otimes 1)
\\&=\bigl(\bar{k_{X\rtimes_\gamma G}}\circ i_X\bigr)\otimes 1
\\&=\OO_{i_X,i_A}\circ k_X\otimes 1
\\&=(\OO_{i_X,i_A}\otimes 1)\circ k_X
\\&=(\bar\Pi\circ i_{\OO_X}\otimes 1)\circ k_X
\\&=\bar{\Pi\otimes\id}\circ(i_{\OO_X}\otimes 1)\circ k_X
\\&=\bar{\Pi\otimes\id}\circ\bar{\what\beta}\circ i_{\OO_X}\circ k_X
\\&=\bar{\bar{\Pi\otimes\id}\circ\what\beta}\circ i_{\OO_X}\circ k_X.
\end{align*}
The verification for generators from $A$ is parallel, using $k_A,i_A,\what\alpha$ instead of $k_X,i_X,\what\gamma$.

For generators from $G$ we have
\begin{align*}
\bar{\zeta\circ\Pi}\circ i_G^{\OO_X}
&=\bar\zeta\circ\bar\Pi\circ i_G^{\OO_X}
\\&=\bar\zeta\circ u
\\&=\bar\zeta\circ \bar{k_{A\rtimes_\alpha G}}\circ i_G^A
\\&=\bar{\zeta\circ k_{A\rtimes_\alpha G}}\circ i_G^A
\\&=\bar{\bar{k_{A\rtimes_\alpha G}\otimes\id}\circ \what\alpha}\circ i_G^A
\\&=\bar{k_{A\rtimes_\alpha G}\otimes\id}\circ\bar{\what\alpha}\circ i_G^A
\\&=\bar{k_{A\rtimes_\alpha G}\otimes\id}\circ\bar{i_G^A\otimes\id}\circ\delta_G
\\&=\bar{\bar{k_{A\rtimes_\alpha G}}\circ i_G^A\otimes\id}\circ\delta_G
\\&=\bar{u\otimes\id}\circ\delta_G
\\&=\bar{\bar\Pi\circ i_G^{\OO_X}\otimes\id}\circ\delta_G
\\&=\bar{\Pi\otimes\id}\circ\bar{i_G^{\OO_X}\otimes\id}\circ\delta_G
\\&=\bar{\Pi\otimes\id}\circ\bar{\what\beta}\circ i_G^{\OO_X}
\\&=\bar{\bar{\Pi\otimes\id}\circ\what\beta}\circ i_G^{\OO_X}.
\end{align*}

Finally, by \cite[Corollary~3.6]{KQRCorrespondenceFunctor}, $\OO_{i_X,i_A}$ is injective because $i_A$ is.
\end{proof}


\begin{thebibliography}{EKQR06}

\bibitem[BS89]{BaajSkandalis}
S.~Baaj and G.~Skandalis, \emph{{$C^*$-alg{\`e}bres de Hopf et th{\'e}orie de
  Kasparov {\'e}quivariante}}, K-Theory \textbf{2} (1989), 683--721.

\bibitem[DCES79]{DeCanniereEnockSchwartz}
J.~De~Canni{\`e}re, M.~Enock, and J.-M. Schwartz, \emph{Alg\`ebres de {F}ourier
  associ\'ees \`a une alg\`ebre de {K}ac}, Math. Ann. \textbf{245} (1979),
  no.~1, 1--22.

\bibitem[DKQar]{dkq}
V.~Deaconu, A.~Kumjian, and J.~Quigg, \emph{Group actions on topological
  graphs}, Ergodic Theory Dynam. Systems \textbf{32} (2012), 1527--1566.

\bibitem[EKQR06]{enchilada}
S.~Echterhoff, S.~Kaliszewski, J.~Quigg, and I.~Raeburn, \emph{{A Categorical
  Approach to Imprimitivity Theorems for C*-Dynamical Systems}}, vol. 180, Mem.
  Amer. Math. Soc., no. 850, American Mathematical Society, Providence, RI,
  2006.

\bibitem[HN08]{HN}
G.~Hao and C.-K. Ng, \emph{Crossed products of {$C^*$}-correspondences by
  amenable group actions}, J. Math. Anal. Appl. \textbf{345} (2008), no.~2,
  702--707.

\bibitem[KQR]{KQRCorrespondenceFunctor}
S.~Kaliszewski, J.~Quigg, and D.~Robertson, \emph{Functoriality of Cuntz-{P}imsner
  correspondence maps}, \url{arXiv:1204.5820}.

\bibitem[Kat84]{kat}
Y.~Katayama, \emph{{Takesaki's duality for a non-degenerate co-action}}, Math.
  Scand. \textbf{55} (1984), 141--151.

\bibitem[Kat04]{KatsuraCorrespondence}
T.~Katsura, \emph{On {$C^*$}-algebras associated with {$C^*$}-correspondences},
  J. Funct. Anal. \textbf{217} (2004), no.~2, 366--401.

\bibitem[Kir77]{KirchbergHopf}
E.~Kirchberg, \emph{Representations of coinvolutive {H}opf-{$W\sp*$}-algebras
  and non-{A}belian duality}, Bull. Acad. Polon. Sci. S\'er. Sci. Math.
  Astronom. Phys. \textbf{25} (1977), no.~2, 117--122.

\bibitem[LPRS87]{lprs}
M.~B. Landstad, J.~Phillips, I.~Raeburn, and C.~E. Sutherland,
  \emph{{Representations of crossed products by coactions and principal
  bundles}}, Trans. Amer. Math. Soc. \textbf{299} (1987), 747--784.

\bibitem[NT79]{NakagamiTakesaki}
Y.~Nakagami and M.~Takesaki, \emph{Duality for crossed products of von
  {N}eumann algebras}, Lecture Notes in Mathematics, vol. 731, Springer,
  Berlin, 1979.

\bibitem[Pim97]{Pi}
M.~V. Pimsner, \emph{A class of {$C^*$}-algebras generalizing both
  {C}untz-{K}rieger algebras and crossed products by {${\bf Z}$}}, Free
  probability theory ({W}aterloo, {ON}, 1995), Fields Inst. Commun., vol.~12,
  Amer. Math. Soc., Providence, RI, 1997, pp.~189--212.

\bibitem[Qui86]{QuiggDualityTwisted}
J.~C. Quigg, \emph{{Duality for reduced twisted crossed products of
  $C^*$-algebras}}, Indiana Univ. Math. J. \textbf{35} (1986), 549--571.

\bibitem[Qui92]{QuiggLandstadDuality}
\bysame, \emph{{Landstad duality for $C^*$-coactions}}, Math. Scand.
  \textbf{71} (1992), 277--294.

\bibitem[RW98]{tfb}
I.~Raeburn and D.~P. Williams, \emph{{Morita equivalence and continuous-trace
  $C^*$-algebras}}, Math. Surveys and Monographs, vol.~60, American
  Mathematical Society, Providence, RI, 1998.

\end{thebibliography}

\providecommand{\bysame}{\leavevmode\hbox to3em{\hrulefill}\thinspace}
\providecommand{\MR}{\relax\ifhmode\unskip\space\fi MR }
\providecommand{\MRhref}[2]{%
  \href{http://www.ams.org/mathscinet-getitem?mr=#1}{#2}
}
\providecommand{\href}[2]{#2}

\end{document}